\documentclass[12pt]{amsart}

\usepackage{amsmath, amsfonts,amsthm,times,graphics}

\usepackage[active]{srcltx}

 \makeatletter
\renewcommand*\subjclass[2][2010]{%
  \def\@subjclass{#2}%
  \@ifundefined{subjclassname@#1}{%
    \ClassWarning{\@classname}{Unknown edition (#1) of Mathematics
      Subject Classification; using '2010'.}%
  }{%
    \@xp\let\@xp\subjclassname\csname subjclassname@#1\endcsname
  }%
}

 \makeatother
\newtheorem{theorem}{Theorem}[section]

\newtheorem{corollary}[theorem]{Corollary}

\newtheorem{identity}[theorem]{Identity}
\newtheorem{congruence}[theorem]{Congruence}
\theoremstyle{definition}
\newtheorem{definition}[theorem]{Definition}

\newtheorem{remark}[theorem]{Remark}

 \makeatletter
\renewcommand*\subjclass[2][2010]{%
  \def\@subjclass{#2}%
  \@ifundefined{subjclassname@#1}{%
    \ClassWarning{\@classname}{Unknown edition (#1) of Mathematics
      Subject Classification; using '1991'.}%
  }{%
    \@xp\let\@xp\subjclassname\csname subjclassname@#1\endcsname
  }%
}

 \makeatother

\begin{document}
\title[Several generalizations and variations of Chu-Vandermonde identity ]
{Several generalizations and variations of\\ Chu-Vandermonde identity}

\author{Romeo Me\v strovi\' c}
\address{Maritime Faculty Kotor, University of Montenegro, 
85330 Kotor, Montenegro} 
\email{romeo@ac.me}

 
\begin{abstract}
In this paper we prove some combinatorial identities which can be considered 
as   generalizations and variations of remarkable Chu-Vandermonde 
identity. These identities are proved by 
using an elementary combinatorial-probabilistic approach   
to the expressions for the  $k$-th moments ($k=1,2,3$) 
of some particular cases of  recently investigated discrete random variables.
Using one of these  Chu-Vandermonde-type identities, two combinatorial  
 congruences are established.  
   \end{abstract}  
  \maketitle

\section{Introduction and Preliminaries}

As noticed in \cite[Section 1.1]{as}, the probabilistic method is a
powerful tool in tackling many problems in Discrete Mathematics
(Combinatorics, Graph Theory, Number Theory and Combinatorial Geometry).
More recently, it has been applied in the development of 
efficient algorithmic techniques and in the study of various computational
problems.

In this paper we present three combinatorial
 identities whose proofs are based on a simple probability  
technique consisting  on calculations of $k$-th moments ($k=1,2,3$) of 
some discrete random variables. Our proofs consist of 
showing that these identities essentially compute the moments 
of order $k$ $(k=1,2,3)$ of the discrete random variable 
defined in \cite{m1}. Notice that this random variable
is a generalization of the complex-valued discrete random variable
defined in \cite{ssa} by providing a statistical analysis for efficient 
detection of signal components when missing data samples are present
(cf. \cite{sso}). 
On the other hand,  the author of this paper continued the 
research on the mentioned  complex-valued discrete random variables
\cite{m8}.

Notice that combinatorial identities and combinatorial problems 
appear in many areas of mathematics, notably in 
Number Theory, Probability Theory, Topology, Geometry, 
Mathematical Optimization, Computer Science, Ergodic Theory and Statistical 
Physics.

As usually, throughout our considerations we use the 
term ``multiset'' (often written as ``set'') to mean ``a totality having possible 
multiplicities''; so that two (multi)sets will be counted as equal if 
and only if they have the same elements with identical multiplicities.
Let $\Bbb C$ and $\Bbb R$ denote the fields of complex and real 
numbers, respectively. For a given positive integer $N$, let ${\mathcal M}_N$ 
denote the collections of all multisets of the form 
   $$
\Phi_N=\{z_1,z_2,\ldots ,z_N:\, z_1,z_2,\ldots ,z_N\in\Bbb C\}.\leqno (1)
   $$
Furthermore, denote by  ${\mathcal M}$  the set consisting 
of all multisets of the form (1), i.e.,  
   $$
{\mathcal M}=\bigcup_{N=1}^{\infty}{\mathcal M}_N. 
  $$

Following Definition 1.2 from \cite{m8} 
(also see Definition 1.1 in \cite{m7}), the random variable $X(m, \Phi_N)$
was generalized in \cite{m1} as follows.
  
  \begin{definition} (\cite[Definition 1.1]{m1})
Let $N$ and $m$  be arbitrary nonnegative integers 
such that  $1\le m\le N$.
For given not necessarily distinct complex numbers
$z_1,z_2,\ldots ,z_N$, let  
$\Phi_N \in {\mathcal M}_N$ be a multiset defined by (1).
 Define the discrete complex-valued random variable $X(m,\Phi_N)$ as
  \begin{eqnarray*}
 && \mathrm{Prob}\left(X(m,\Phi_N)
 = \sum_{i=1}^m z_{n_i}\right)\\
(2)\qquad\qquad &= &\frac{1}{{N\choose m}}\cdot \big|\{\{t_1,t_2,\ldots,t_m\}
\subset\{1,2,\ldots,N\}: 
\sum_{i=1}^m z_{t_i}=\sum_{i=1}^mz_{n_i}\} \big|
\quad\qquad\\
& =&:\frac{q(n_1,n_2,\ldots, n_m)}{{N\choose m}},
  \end{eqnarray*}
  where $\{n_1,n_2,\ldots, n_m\}$ is an arbitrary fixed
 subset of $\{1,2,\ldots,N\}$ such that $1\le n_1<n_2<\cdots <n_m\le N$;
moreover, $q(n_1,n_2,\ldots, n_m)$ is the cardinality of a collection
 of all subsets
$\{t_1,t_2,\ldots,t_m\}$ of the set $\{1,2,\ldots,N\}$ such that 
$\sum_{i=1}^m z_{t_i}=\sum_{i=1}^mz_{n_i}$.
  \end{definition}

Notice that the above definition is correct  taking into account that there 
are ${N\choose m}$ index sets $T\subset \{1,2,\ldots,N \}$ with 
$m$ elements. Moreover, a very  short, but not 
strongly exact version of Definition 1.1  is given as follows 
(cf. \cite[Definition 1.2']{m8}).
   \vspace{2mm}

  \noindent {\bf Definition 1.1'.}
Let $N$ and $m$  be arbitrary nonnegative integers 
such that  $1\le m\le N$.
For given not necessarily distinct complex numbers
$z_1,z_2,\ldots ,z_N$, let  
$\Phi_N \in {\mathcal M}_N$ be a multiset defined by (1).
Choose a random subset  $S$ of size $m$ (the so-called $m$-element 
subset) without replacement from the set 
$\{1,2,\ldots, N\}$. Then the 
 complex-valued discrete random variable $X(m,\Phi_N)$ is defined as a sum
  $$
X(m,\Phi_N)=\sum_{n\in S}z_n.
  $$
  \vspace{1mm}

It was proved in \cite{m1}  the following result 
(cf. \cite[proof of Theorem 2.1]{m8}  as a particular case).

 \begin{theorem}$($\cite[the expressions (3) and (5) of Theorem 1.2]{m1}$).$
Let $N$ and $m$ be positive integers such that
  $N\ge 2$ and  $1\le m\le N$. Let $\Phi_N=\{z_1,z_2,\ldots ,z_N\}$ be any 
multiset with $z_1,z_2,\ldots ,z_N\in \Bbb C$. 
 Then the expected value of the random variable $X(m,\Phi_N)$  
from Definition $1.1$ and the second moment of the random variable 
$|X(m,\Phi_N)|$  are respectively given  by
 by
  $$
\Bbb E[X(m,\Phi_N)]=\frac{m}{N}\sum_{i=1}^Nz_i,\leqno(3) 
   $$
and
      $$
   \Bbb E[|X(m,\Phi_N)|^2]=
\frac{m}{N(N-1)}\left((N-m) \sum_{i=1}^N |z_i|^2 + 
 (m-1)\big| \sum_{i=1}^N z_i \big|^2\right).\leqno(4)
    $$
 \end{theorem}

Notice that in the case when $X(m,\Phi_N)$ is a 
real-valued random variable (i.e., if in Definition 1.1 
$z_1,z_2,\ldots,z_n$ are real numbers), then the following  expression for 
the third moment $\Bbb E[(X(m,\Phi_N))^3]$ of $X(m,\Phi_N)$ can be
proved similarly as the above expression (4) given in \cite[Theorem 2.1]{m1}.

 \begin{theorem} 
Let $N$ and $m$ be positive integers such that
  $N\ge 2$ and let $1\le m\le N$. Suppose that  $z_1,z_2\ldots,z_N$ 
are not necessarily distinct real numbers, and let $X(m,\Phi_N)$
be the real-valued random variable from Definition $1.1$.  Then 

\begin{eqnarray*}
&&   \Bbb E[(X(m,\Phi_N))^3]\\ 
(5)\qquad\qquad &&=
\frac{m}{N(N-1)}\left((N+2-3m) \sum_{i=1}^N
 z_i^3 + 3(m-1) (\sum_{i=1}^N z_i^2)
(\sum_{i=1}^N z_i)\right).
    \end{eqnarray*}

 \end{theorem}
 \begin{proof}[Proof of Theorem $1.3$]
By Definition 1.1 and by the definition of the third moment 
$\Bbb E[(X(m,\Phi_N))^3]$, we find that
    $$
\Bbb E[(X(m,\Phi_N))^3]=\frac{1}{{N\choose m}}\sum_{\{i_1,i_2,\ldots,i_m\} 
\subset \{1,2,\ldots,N\}}(z_{i_1}+z_{i_2}+\cdots+z_{i_m})^3,\leqno(6)
   $$
where the summation ranges over all subsets $\{i_1,i_2,\ldots,i_m\}$
of $\{1,2,\ldots,N\}$ with $1\le i_1<i_2<\cdots <i_m\le N$.
Since for any fixed  $s\in\{1,2,\ldots,N\}$,  $z_{s}$  occurs 
exactly ${N-1\choose m-1}$ times  in the expanded  sum on the right hand side 
of (6), using the multinomial formula for 
$(x_1+x_2+\cdots +x_s)^3$ (see, e.g., \cite[p. 28]{st}), from (6)  we obtain
     \begin{equation*}\begin{split}
&\Bbb E[(X(m,\Phi_N))^3]=\frac{1}{{N\choose m}}\left({N-1\choose m-1}
\sum_{i=1}^Nz_i^3
+ 3{N-2\choose m-2}\sum_{1\le i<n\le N}z_i^2z_n\right)\\
(7)\quad =&
\frac{1}{{N\choose m}}\left(\Big({N-1\choose m-1}-3{N-2\choose m-2}\Big) 
\sum_{i=1}^Nz_i^3\qquad\qquad\right.\qquad\qquad\qquad\qquad\\
& \left. +3{N-2\choose m-2}\Big(\sum_{i=1}^Nz_i^3+\sum_{1\le k<n\le N}z_k^2
z_n\Big)\right)\\
=&\frac{1}{{N\choose m}}\left(\Big({N-1\choose m-1}-3{N-2\choose m-2}\Big) 
\sum_{i=1}^Nz_i^3  +3{N-2\choose m-2}\sum_{1\le k<n\le N}^Nz_k^2z_n\right).\\
=&\frac{1}{{N\choose m}}\left(\Big({N-1\choose m-1}-3{N-2\choose m-2}\Big) 
\sum_{i=1}^Nz_i^3  +3{N-2\choose m-2}\big(\sum_{i=1}^Nz_i^2\big)
\big(\sum_{i=1}^Nz_i\big)\right).    
   \end{split}\end{equation*}
Then replacing the identities  
${N\choose m}=\frac{N}{m}{N-1\choose m-1}$ and 
${N\choose m}=\frac{N(N-1)}{m(m-1)}{N-2\choose m-2}$ into (7),
a routine calculation yields the expression (5).  
  \end{proof}

It was indicated in \cite[Section 3]{m8} that for some particular cases 
of sets $\Phi(N)$ (given by (1)) and some values $m$, the expressions 
(3) and   (4)  concerning the associated random variables $X(m,\Phi_N)$ 
yield some combinatorial identities. For a comprehensive  list 
of  combinatorial identities see \cite{go} and \cite{go2} 
(also see \cite{ri} and \cite[Chapter 5]{gkp}).
Motivated by this fact, by using some other particular cases 
of the random variables $X(m,\Phi_N)$ from Definition 1.1,  
in the next section  we deduce some new and some known  
combinatorial identities which can be considered as 
generalizations of Chu-Vandermonde identity. Notice that 
Chu-Vandermonde identity is often called  Vandermonde's identity
or sometimes Vandermonde's formula.   

\section{Chu-Vandermonde-type identities and their proofs}

We start with the following identity. 

   \begin{identity}
Let $n_1,n_2,\ldots ,n_s$ be arbitrary positive 
integers and le $z_1,z_2,\ldots ,z_s$ be arbitrary complex numbers
    $(s\ge 2)$. If $m$ is a a positive integer such that 
$m\le \sum_{i=1}^sn_i$, then
   \begin{eqnarray*}
&& \sum_{\sum_{i=1}^sk_i=m\atop k_1\le n_1,k_2\le n_2,\ldots, 
k_s\le n_s}{n_1\choose k_1}
{n_2\choose k_2}\cdots {n_s\choose k_s}(k_1z_1+k_2z_2+\cdots+k_sz_s)\\
(8)\qquad\qquad &=&{\sum_{i=1}^s n_i\choose m}\frac{m(\sum_{i=1}^s n_iz_i)}{\sum_{i=1}^s n_i}
 ,\qquad\qquad
   \end{eqnarray*}
 where the summation ranges over all nonnegative integers 
$k_i$  $(i=1,2,\ldots , s)$ such that $k_1\le n_1,k_2\le n_2,\ldots, 
k_s\le n_s$ and  $\sum_{i=1}^sk_i=m$.
  \end{identity}

\begin{proof} Put $\sum_{i=1}^sn_i=N$ and consider the multiset $\Phi_N$  
defined by
    $$
\Phi_N =\{\underbrace{z_1,\ldots,z_1}_{n_1},
\underbrace{z_2,\ldots,z_2}_{n_2},\ldots, \underbrace{z_s,\ldots,z_s}_{n_s}\}.
  $$
Now consider the random variable  $X(m,\Phi_N)$ given by Definition 1.1.
Then by the expression (3) of Theorem 1.2, we have 
  $$
\Bbb E[X(m,N)]=\frac{m(\sum_{i=1}^s n_iz_i)}{\sum_{i=1}^s n_i}.\leqno(9)
   $$  
On the other hand, for each $s$-tuple $(k_1,k_2,\ldots,k_s)$ of
nonnegative integers $k_1,k_2,\ldots,k_s$ such that $k_1\le n_1,k_2\le n_2,\ldots, 
k_s\le n_s$ and  $\sum_{i=1}^sk_i=m$, by (3), we get
   $$
  \mathrm{Prob}\left(X(m,\Phi_N) = \sum_{i=1}^s k_iz_i\right)
=  \frac{1}{{\sum_{i=1}^s n_i\choose m}}
{n_1\choose k_1}{n_2\choose k_2}\cdots {n_s\choose k_s}.\leqno(10)
   $$
Then by definition of expectation of a complex-valued discrete random variable,
from (10) we find that 
   \begin{eqnarray*}
 &&\Bbb E[X(m,N)]\\
(11)\quad &=&\frac{1}{{\sum_{i=1}^s n_i\choose m}}
\sum_{\sum_{i=1}^sk_i=m\atop k_1\le n_1,k_2\le n_2,\ldots, 
k_s\le n_s}{n_1\choose k_1}
{n_2\choose k_2}\cdots {n_s\choose k_s}(k_1z_1+k_2z_2+\cdots+k_sz_s),
  \end{eqnarray*}
where the summation ranges over all nonnegative integers 
$k_i$  ($i=1,2,\ldots , s)$ such that $k_1\le n_1,k_2\le n_2,\ldots, 
k_s\le n_s$ and  $\sum_{i=1}^sk_i=m$. 

Finally, comparing the equalities (9) and (11), we immediately obtain (8). 
 \end{proof}

\begin{remark} Quite recently, by using functional equations of 
the generating function of certain class of polynomials, a new 
 Chu-Vandermonde-type identity (Vandermonde type convolution 
formula) is derived  in \cite[Theorem 5.4 of Section 5]{ks}. 
As a special  case of this result is the 
following identity \cite[Corollary 5.5 of Section 5]{ks}:
      $$
\sum_{v_1=0}^n{k_1+v_1-1\choose v_1}{k_2+n-v_1-1 \choose n-v_1}=
{k_1+k_2+n-1\choose n},
     $$   
where $k_1\ge 1$, $k_2\ge 1$ and $n$ are nonnegative integers.
Another generalization of Chu-Vandermonde identity was recently 
given in \cite{kr}. Moreover, two diferent interpretations 
of this identity are recently  considered in \cite{so}: 
as an identity for polynomials, and as an identity for infinite matrices.  
   \end{remark}

  \begin{remark}
If $P_{s-1}(z)=\sum_{i=1}^sn_iz^{i-1}$ is a complex polynomial
of the variable  $z$ of degree $s-1$  with integer coefficients 
$n_1,n_2,\ldots ,n_s\ge 0$ ($n_s\not= 0$), then taking $z_i=z^{i-1}$   
($i=1,2,\ldots,s$) into the identity (8), it becomes 
   \begin{equation*}\begin{split}
& \sum_{\sum_{i=1}^sk_i= m\atop k_1\le n_1,k_2\le n_2,\ldots, 
k_s\le n_s}{n_1\choose k_1}
{n_2\choose k_2}\cdots {n_s\choose k_s}(k_1+k_2z+\cdots+k_sz^{s-1}) \\
& =\frac{m}{\sum_{i=1}^s n_i} {\sum_{i=1}^s n_i\choose m}P_{s-1}(z),
    \end{split}\end{equation*}
 where the summation ranges over all nonnegative integers 
$k_i$  ($i=1,2,\ldots , s)$ such that $k_1\le n_1,k_2\le n_2,\ldots, 
k_s\le n_s$ and  $\sum_{i=1}^sk_i=m$.
   \end{remark}

\begin{remark} As usually, if we use the convention that  ${a\choose b}=0$ for 
all nonnegative integers $a$ and $b$ such that $b>a$, then
the conditions $k_1\le n_1,k_2\le n_2,\ldots,k_s\le n_s$  
which appear under the first sum $\sum\cdot$ of (8) can be omitted.
 \end{remark}

A particular case of  Identity 2.1 is Identity 2.5 given below
which is  a well known ``multinomial''  generalization of the 
{Vandermonde identity} (often called  
{\it Vandermonde convolution formula} or {\it Chu-Vandermonde convolution}) 
(see, e.g., \cite{ri}).

\begin{identity} Let $n_1,n_2,\ldots ,n_s$  $(s\ge 2)$ be arbitrary positive 
integers. If $m$ is a positive integer $m$ such that 
$m\le \sum_{i=1}^sn_i$, then
   $$
\sum_{\sum_{i=1}^sk_i=m\atop k_1\le n_1,k_2\le n_2,\ldots, 
k_s\le n_s}{n_1\choose k_1}{n_2\choose k_2}\cdots 
{n_s\choose k_s}={\sum_{i=1}^s n_i\choose m},\leqno(12)
   $$  
where the summation ranges over all nonnegative integers 
$k_i$  $(i=1,2,\ldots , s)$ such that $k_1\le n_1,k_2\le n_2,\ldots, 
k_s\le n_s$ and  $\sum_{i=1}^sk_i=m$.  
  \end{identity}
  
\begin{proof}
Taking  $z_1=z_2=\cdots =z_s=1$ into equality (8), we immediately obtain (12).
\end{proof}
   
\begin{remark} There are well known  algebraic  and  combinatorial proofs
 of the  identity (12) (see, e.g., \cite{wi}).  Notice also that 
 for $s=2$ the identity (12) with $k_1=k$, $n_1=a$
and $n_2=b$  simplifies to Chu-Vandermonde identity given by 
(see, e.g., \cite[p. 67]{aa})
     $$
\sum_{k=0}^m{a\choose k}{b\choose m-k}={a+b\choose m},\leqno(13)
    $$
which also holds for any complex numbers $a$ and $b$.  
Notice that
the identity (13) is named after A.T. Vandermonde (1772), although it was already 
known in 1303 by the Chinese mathematician Zhu Shijie (Chu Shih-Chieh)
(see \cite[pp. 59--60]{ask} for the history).
This identity plays an important role in Combinatorics,
Combinatorial Number Theory and Probability Theory (\cite{gkp}, 
\cite{gs} and \cite{ri}). As indicated in \cite[p. 8]{ri},
 Vandermonde convolution formula 
is perhaps the most widely used combinatorial identity. 
In the literature there are many proofs of this identity and its 
several generalizations. A  proof given in \cite{vm} was established by giving 
probabibilstic interpretations to the summands.    
  \end{remark}

Taking $s=2$ and $z_1/z_2=z$ ($z_2\not= 0$) into (8), it 
simplifies to the  following Vandermonde-type convolution formula.

\begin{identity}
Let $n_1$ and $n_2$ be arbitrary positive 
integers and let $z$ be  any  complex numbers. 
If $m$ is a positive integer  such that 
$m\le n_1+n_2$, then
     $$
  \sum_{k=0}^m{n_1\choose k}{n_2\choose m-k}(kz+(m-k))
={n_1+n_2\choose m}\frac{2m(n_1z+n_2)}{n_1+n_2}.\leqno(14)
   $$
\end{identity}

\begin{remark} Observe that taking $z=1$, $n_1=a$ and 
$n_2=b$ into (14), it immediately reduces to 
Chu-Vandermonde identity given by (13). 
  \end{remark}

Substituting $s=3$ $z_1/z_3=z$ and $z_2/z_3=w$ $(z_3\not= 0)$  into (8), it 
reduces to the  following Vandermonde-type convolution formula.

\begin{identity}
Let $n_1$, $n_2$ and $n_3$ be arbitrary positive 
integers and let $z$ and $w$ be  arbitrary  complex numbers. 
If $m$ is a positive integer such that 
$m\le n_1+n_2+n_3$, then
  \begin{eqnarray*}
 &&  \sum_{0\le k_1+k_2\le m\atop k_1\ge 0,k_2\ge 0}{n_1\choose k_1}
{n_2\choose k_2} {n_3\choose m-k_1-k_2}(k_1z+k_2w+(m-k_1-k_2))\qquad\\
(15)\qquad &=&
{n_1+n_2+n_3\choose m}\frac{2m(n_1z+n_2w+n_3)}{n_1+n_2+n_3}.\qquad\qquad
   \end{eqnarray*}
\end{identity}

Taking $z=w=1/2$ into (15), we obtain the following identity.

\begin{identity}
Let $n_1$, $n_2$ and $n_3$ be arbitrary positive 
integers and let $z_1$ and $z_2$ be  arbitrary  complex numbers. 
If $m$ is a positive integer such that 
$m\le n_1+n_2+n_3$, then
  \begin{eqnarray*}
 &&  \sum_{0\le k_1+k_2\le m\atop k_1\ge 0,k_2\ge 0}(2m-k_1-k_2){n_1\choose k_1}
{n_2\choose k_2} {n_3\choose m-k_1-k_2}\\
(16)\qquad &=&
\frac{2m(n_1+n_2+2n_3)}{n_1+n_2+n_3}{n_1+n_2+n_3\choose m}.\qquad\qquad\qquad\qquad\qquad\qquad\qquad
   \end{eqnarray*}
\end{identity}

Another special case of  Identity 2.1 is given as follows.

\begin{identity} 
Let $s$ and $l$ be arbitrary positive integers, and let 
$m$ be a positive integer such that $m\le sl$. Then
 $$
\sum_{\sum_{i=1}^sk_i=m \atop k_1\le l,k_2\le l,\ldots, 
k_s\le l}{l\choose k_1}
{l\choose k_2}\cdots {l\choose k_s}(k_1+2k_2+\cdots+sk_s)=\frac{m(s+1)}{2}
 {sl\choose m},\leqno(17)
   $$
where the summation ranges over all nonnegative integers 
$k_i$  $(i=1,2,\ldots , s)$ such that $\sum_{i=1}^sk_i=m$ and  
$k_1\le l,k_2\le l,\ldots,k_s\le l$.
  \end{identity}

  \begin{proof}
Substituting  $z_i=i$ for all $i=1,2,\ldots ,s$ and $n_1=n_2=\cdots =n_s=l$
into (8), it immediately reduces to the equality (17). 
 \end{proof}

As a consequence of Identity 2.11, we obtain the following 
``supercongruence'' closely related to the remarkable 
{\it Wolstenhlme's theorem} which asserts that
  $$
{2p-1\choose p-1}\equiv 1\pmod{p^3}
 $$ 
for all primes $p\ge 3$ (\cite{wo}; also see \cite[p. 3]{m9} and \cite{m2}).

\begin{congruence} 
Let $p\ge 5$ be a prime. Then for each positive integer $s$ there holds 
  $$
\sum_{\sum_{i=1}^sk_i=p\atop k_1,k_2,\ldots, k_s\ge 0}{p\choose k_1}
{p\choose k_2}\cdots {p\choose k_s}(k_1+2k_2+\cdots+sk_s)
\equiv \frac{s(s+1)p}{2}\pmod{p^4}.\leqno(18)
    $$
In particular,
 $$
\sum_{\sum_{i=1}^sk_i=p\atop k_1,k_2,\ldots, k_s\ge 0}{p\choose k_1}
{p\choose k_2}\cdots {p\choose k_s}(k_1+2k_2+\cdots+sk_s)\equiv 0 \pmod{p}.
\leqno(19)
    $$
    \end{congruence}
\begin{proof} 
If we substitute $l=m=p$ into equality (17), then its right hand
side is equal to $\frac{(s+1)p}{2}{sp\choose p}$. Since by 
Glaisher's congruence \cite[p. 21]{gl} 
(or more general, Ljunggren's congruence (\cite{lj}; also see 
\cite[the congruences (15), p. 7 and (35) and (36), p. 11]{m9}; 
cf. \cite[Section 3.3]{m3} and \cite{m10}), for any prime $p\ge 5$ and 
a positive integer $s$, we have 
   $$
{sp\choose p}\equiv s\pmod{p^3}, 
   $$   
and hence, 
   $$
\frac{(s+1)p}{2}{sp\choose p}\equiv \frac{s(s+1)p}{2}\pmod{p^4}.
  $$
Substituting the  above congruence into (17) with $l=m=p$,
we immediately obtain the congruence (18). Finally, reducing the modulus 
in (18) to $(\bmod{\, p})$, implies (19).   
  \end{proof}

Let us recall that a prime $p$ is said to be a {\it Wolstenholme prime} 
(see, e.g., \cite{mc},   \cite[Section 7]{m9} and \cite{m11}; this is 
Sloane's sequence A088164 from \cite{sl}) if it satisfies the congruence 
  $$
{2p-1\choose p-1}\equiv 1\pmod{p^4}.
  $$  
It is well known (see \cite[p. 21]{gl}, \cite[p. 323]{gl2} and 
\cite[p. 14]{m9}) that $p$ is a Wolstenholme prime if and only if $p$ divides
the numerator of the {\it Bernoulli number} $B_{p-3}$.
Moreover, these primes together with the primes $p$ such that 
the  Euler number $E_{p-3}$ is divisible by $p$, are closely
related to  the first case of {\it Fermat Last Theorem} 
(see \cite{va} and \cite{m3}). It can be shown that for any Wolstenholme 
prime, the congruence (18) holds modulo $p^5$, i.e., we have the following
assertion. 
\begin{congruence} 
Let $p$ be a Wolstenholme prime. Then for each positive integer $s$ 
there holds 
  $$
\sum_{\sum_{i=1}^sk_i=p\atop k_1,k_2,\ldots, k_s\ge 0}{p\choose k_1}
{p\choose k_2}\cdots {p\choose k_s}(k_1+2k_2+\cdots+sk_s)\equiv \frac{s(s+1)p}{2}
\pmod{p^5}.\leqno(20)
    $$
    \end{congruence}
\begin{proof} Notice that by a result of Glaisher (\cite[p. 21]{gl}, 
\cite[p. 323]{gl2}; also see \cite[the conguence (15), p. 7]{m9}
and the  {\it Jacobsthal's congruence} \cite{lj}),
for any Wolstenholme prime $p$,
  $$
{sp\choose p}\equiv s\pmod{p^4}. 
   $$
Then the rest of the  proof is quite similar to that of the previous 
Congruence 2.12, and hence may be omitted.
   \end{proof}

Another consequence of Identity 2.1 is given as follows.
  \begin{identity}
Let $n\ge 2$ and $s$  be  fixed positive integers and let 
$k=k_1+k_2n+\cdots +  k_sn^{s-1}$ be the base $n$ 
representation of a positive integer $k<n^s$ 
$($with $0\le k_1,k_2,\ldots, k_s\le n-1$$)$. 
If $m$ is a positive integer such that $m\le s(n-1)$,  then
  \begin{equation*}\begin{split}
& \sum_{\sum_{i=1}^sk_i= m\atop k_1\le n-1,k_2\le n-1,\ldots, 
k_s\le n-1}{n-1\choose k_1}
{n-1\choose k_2}\cdots {n-1\choose k_s}(k_1+k_2n+\cdots+k_sn^{s-1}) \\
(21) & =(n^s-1){s(n-1)-1\choose m-1}.
    \end{split}\end{equation*}
  \end{identity}
 \begin{proof}
Setting $n_1=n_2=\cdots =n_s=n-1$ and $z_i=n^{i-1}$ 
($i=1,2,\ldots, s$) into the identity (8) and using the identity 
${s(n-1)\choose m}=\frac{s(n-1)}{m}{s(n-1)-1\choose m-1}$, immediately gives 
the identity (21).   
\end{proof}

The binary case of Identity 2.14 can be reformulated as follows. 

  \begin{corollary}
Let $s$ and $m$  be  positive integers 
such that $m\le s$.  Then the sum of all positive integers 
less than $2^s$ whose binary representation contains exactly 
$m$  $1'$s is equal to $(2^s-1){s-1\choose m-1}$ 
$($as usually, it is assumed that ${0\choose 0}=1$$)$.
 \end{corollary}
\begin{proof}
Taking $n=2$ into (21), we have 
  $$
\sum_{\sum_{i=1}^sk_i= m\atop k_1,k_2,\ldots, 
k_s\in\{0,1\}}(k_1+2k_2+\cdots+2^{s-1}k_s)=(2^s-1){s-1\choose m-1}.
  $$
The above identity is in fact the assertion of the corollary.  
  \end{proof}

\begin{remark} Note that Corollary 2.15 can be easily proved by 
induction on $s\ge 1$ and also by using a simple counting argument.
  \end{remark}

A quadratic analogue of Identity 2.1 is given as follows.

\begin{identity}
Let $n_1,n_2,\ldots ,n_s$ be arbitrary positive 
integers and let $z_1,z_2,\ldots ,z_s$ be arbitrary complex numbers
    $(s\ge 2)$. If $m$ is a  positive integer  such that 
$2\le m\le \sum_{i=1}^sn_i$, then
   \begin{eqnarray*}
 &&\sum_{\sum_{i=1}^sk_i=m\atop k_1\le n_1,k_2\le n_2,\ldots, 
k_s\le n_s}{n_1\choose k_1}
{n_2\choose k_2}\cdots {n_s\choose k_s}|k_1z_1+k_2z_2+\cdots+k_sz_s|^2\\
(22)\qquad &=&{\sum_{i=1}^s n_i -2\choose m-1}\sum_{i=1}^s n_i|z_i|^2
+{\sum_{i=1}^s n_i -2\choose m-2}|\sum_{i=1}^s n_iz_i|^2,
 \qquad\qquad\qquad
   \end{eqnarray*}
 where the summation ranges over all nonnegative integers 
$k_i$  $(i=1,2,\ldots , s)$ such that $k_1\le n_1,k_2\le n_2,\ldots, 
k_s\le n_s$ and  $\sum_{i=1}^sk_i=m$.
  \end{identity}
  
\begin{proof} Put $\sum_{i=1}^sn_i=N$ and as in the proof of Identity 2.1,
 consider the multiset $\Phi_N$ defined by
    $$
\Phi_N =\{\underbrace{z_1,\ldots,z_1}_{n_1},
\underbrace{z_2,\ldots,z_2}_{n_2},\ldots, \underbrace{z_s,\ldots,z_s}_{n_s}\}.
  $$
Now consider the random variable  $X(m,\Phi_N)$ given by Definition 1.1.
Then by the expression (4) of Theorem 1.2, we have 
  $$
\Bbb E[|X(m, \Phi_N)|^2]=\frac{m(N-m)}{N(N-1)}   
\big(\sum_{i=1}^s n_i|z_i|^2
+\frac{m(m-1)}{N(N-1)}|\sum_{i=1}^s n_iz_i|^2\big).\leqno(23)
   $$  
On the other hand, for each $s$-tuple $(k_1,k_2,\ldots,k_s)$ of
nonnegative integers $k_1,k_2,\ldots,k_s$ such that $k_1\le n_1,k_2\le n_2,
\ldots,k_s\le n_s$ and  $\sum_{i=1}^sk_i=m$, by (2) we get
   $$
\mathrm{Prob}\left(X(m,\Phi_N) = \sum_{i=1}^s k_iz_i\right)
=  \frac{1}{{N\choose m}}
{n_1\choose k_1}{n_2\choose k_2}\cdots {n_s\choose k_s}.\leqno(24)
   $$
Then by definition of the expectation of a  discrete random 
variable, from (24) we find that 
   \begin{eqnarray*}
 &&\Bbb E[|X(m,\Phi_N)|^2]\\
(25)\qquad &=&\frac{1}{{N\choose m}}\sum_{\sum_{i=1}^sk_i=m\atop k_1\le n_1,k_2\le n_2,\ldots, 
k_s\le n_s}{n_1\choose k_1}
{n_2\choose k_2}\cdots {n_s\choose k_s}|k_1z_1+k_2z_2+\cdots+k_sz_s|^2,
  \end{eqnarray*}
where the summation ranges over all nonnegative integers 
$k_i$  ($i=1,2,\ldots , s)$ such that $k_1\le n_1,k_2\le n_2,\ldots, 
k_s\le n_s$ and  $\sum_{i=1}^sk_i=m$. 

Note that if  $m=\sum_{i=1}^sn_i:=N$, then
the  sum on the left hand side of (22) consists of one term
(with $k_1=n_1,k_2=n_2,\ldots , k_s=n_s$) equals to 
$|\sum_{i=1}^sn_iz_i|^2$, which is (because of 
${m-2\choose m-1}=0$) identically equal to the the right hand side of (22).  
   Finally, if $m\le N-1$, then comparing the equalities (23) and (25), using the 
identities ${N\choose m}=\frac{N(N-1)}{m(m-1)}{N-2\choose m-2}$
and ${N\choose m}=\frac{N(N-1)}{m(N-m)}{N-2\choose m-1}$,
 we immediately obtain (22). 
    \end{proof}
In particular, Identity 2.17 implies the following one. 

\begin{identity} Let $s$ and $l$ be arbitrary positive integers, and let 
$m$ be a positive integer such that $m\le sl$. Then
    \begin{eqnarray*}
 &&
\sum_{\sum_{i=1}^sk_i=m \atop k_1\le l,k_2\le l,\ldots, 
k_s\le l}{l\choose k_1}
{l\choose k_2}\cdots 
{l\choose k_s}(k_1+2k_2+\cdots+sk_s)^2\\
(26)\qquad\qquad &=&
\frac{m(s+1)(3s^2lm+3slm+s^2l-4sm-sl-2m)}{12(sl-1)}
 {sl\choose m},\qquad\qquad
   \end{eqnarray*}
where the summation ranges over all nonnegative integers 
$k_i$  $(i=1,2,\ldots , s)$ such that $\sum_{i=1}^sk_i=m$ and  
$k_1\le l,k_2\le l,\ldots,k_s\le l$.
\end{identity}
  \begin{proof}
Substituting  $z_i=i$ for all $i=1,2,\ldots ,s$ and $n_1=n_2=\cdots =n_s=l$
into (22), and taking $\sum_{i=1}^s i=s(s+1)/2$ and 
$\sum_{i=1}^s i^2=s(s+1)(2s+1)/6$, it immediately reduces to the identity (26). 
 \end{proof}

We also  give a real cubic analogue of Identity 2.1 as follows.

\begin{identity}
Let $n_1,n_2,\ldots ,n_s$ be arbitrary  positive 
integers and let  $z_1,z_2,\ldots ,z_s$ be arbitrary
 real  numbers
    $(s\ge 2)$. If $m$ is a positive integer such that 
$m\le \sum_{i=1}^sn_i$, then
   \begin{eqnarray*}
 &&\sum_{\sum_{i=1}^sk_i=m\atop k_1\le n_1,k_2\le n_2,\ldots, 
k_s\le n_s}{n_1\choose k_1}
{n_2\choose k_2}\cdots {n_s\choose k_s}(k_1z_1+k_2z_2+\cdots+k_sz_s)^3\\
(27) &=&
\left(\frac{m(N-m)}{N(N-1)} -\frac{2m(m-1)}{N(N-1)} \right) \sum_{i=1}^s
 n_iz_i^3 + \frac{3m(m-1)}{N(N-1)} (\sum_{i=1}^sn_iz_i^2)(\sum_{i=1}^sn_iz_i).
  \end{eqnarray*}
  \end{identity}
  
\begin{proof} The proof of the identity (27) is based on the 
expression (5) of  Theorem 1.3, 
and since it is quite similar to that of Identity 2.17,  may be omitted. 
 \end{proof}

\begin{remark}
Consider the identity (27) as the identity in the 
ring $\Bbb R[z_1,z_2,\ldots, z_s]$ of real  polynomials 
in $s$ variables. Then by induction on $s\ge 1$, it can be easily 
to show that the identity (27) also holds in the case 
when $z_1,z_2,\ldots, z_s$ are arbitrary complex numbers.  
  \end{remark}

Furthermore, notice that by linearity Identity 2.1 can be immediately extended 
in matrix form as follows.

   \begin{identity} Denote by ${\Bbb K}^{M\times N}$ the vector  space of 
all matrices over the field   $\Bbb K$ with $M$ rows and $N$ columns
$(\Bbb K=\Bbb C$ or $\Bbb K=\Bbb R$ and $M,N\ge 1)$. 
Let $n_1,n_2,\ldots ,n_s$ be arbitrary positive 
integers and let 
${\rm{\bf A}}_1,{\rm{\bf A}}_2,\ldots ,{\rm{\bf A}}_s\in {\Bbb K}^{M\times N}$ 
be arbitrary $M\times N$ matrices. If $m$ is a  positive integer such that 
$m\le \sum_{i=1}^sn_i$, then
   \begin{eqnarray*}
&& \sum_{\sum_{i=1}^sk_i=m\atop k_1\le n_1,k_2\le n_2,\ldots, 
k_s\le n_s}{n_1\choose k_1}
{n_2\choose k_2}\cdots {n_s\choose k_s}(k_1{\rm{\bf A}}_1+k_2{\rm{\bf A}}_2+
\cdots+k_s{\rm{\bf A}}_s)\qquad\\
(28)\qquad &=&{\sum_{i=1}^s n_i\choose m}\frac{m(\sum_{i=1}^s 
n_i{\rm{\bf A}}_i)}{\sum_{i=1}^s n_i},\qquad\qquad
   \end{eqnarray*}
 where the summation ranges over all nonnegative integers 
$k_i$  $(i=1,2,\ldots , s)$ such that $k_1\le n_1,k_2\le n_2,\ldots, 
k_s\le n_s$ and  $\sum_{i=1}^sk_i=m$.
  \end{identity}

Finally, we believe that using the counting method applied in \cite{m12}, it can be 
proved the following two   matrix analogues of Identity 2.17 (which 
are verified for some  small values of $s$, $n_1,n_2,\ldots ,n_s$ and $m$).
\begin{identity}
Denote by ${\Bbb K}^{N\times N}$ the algebra of 
all square  matrices of order $N$ $(N\ge 1)$ over the field
$\Bbb K=\Bbb C$ or $\Bbb K=\Bbb R$.  Let $n_1,n_2,\ldots ,n_s$ be arbitrary positive 
integers and let 
${\rm{\bf A}}_1,{\rm{\bf A}}_2,\ldots ,{\rm{\bf A}}_s\in {\Bbb K}^{N\times N}$ 
be arbitrary square  matrices of order $N$. If $m$ is a  positive integer such that 
$m\le \sum_{i=1}^sn_i$, then
    \begin{eqnarray*}
 &&\sum_{\sum_{i=1}^sk_i=m\atop k_1\le n_1,k_2\le n_2,\ldots, 
k_s\le n_s}{n_1\choose k_1}
{n_2\choose k_2}\cdots {n_s\choose k_s}(k_1{\rm{\bf A}}_1 +k_2{\rm{\bf A}}_2+
\cdots+k_s{\rm{\bf A}}_s)^2\\
(29)\qquad &=&{\sum_{i=1}^s n_i -2\choose m-1}\sum_{i=1}^s n_i
{\rm{\bf A}}_i^2+{\sum_{i=1}^s n_i -2\choose m-2}(\sum_{i=1}^s 
n_i{\rm{\bf A}}_i)^2,
 \qquad\qquad\qquad
   \end{eqnarray*}
 where the summation ranges over all nonnegative integers 
$k_i$  $(i=1,2,\ldots , s)$ such that $k_1\le n_1,k_2\le n_2,\ldots, 
k_s\le n_s$ and  $\sum_{i=1}^sk_i=m$.
  \end{identity}

\begin{identity} Denote by ${\Bbb K}^{M\times N}$ the vector  space of 
all matrices over the field   $\Bbb K$ with $M$ rows and $N$ columns
$(\Bbb K=\Bbb C$ or $\Bbb K=\Bbb R$ and $M,N\ge 1)$.
Let ${\rm{\bf A}}^*\in{\Bbb K}^{N\times M}$ be  the conjugate transpose $($Hermitian transpose$)$
of a matrix  ${\rm{\bf A}}\in{\Bbb K}^{M\times N}$. 
  Let $n_1,n_2,\ldots ,n_s$ be arbitrary positive 
integers and let 
${\rm{\bf A}}_1,{\rm{\bf A}}_2,\ldots ,{\rm{\bf A}}_s\in {\Bbb K}^{M\times N}$ 
be arbitrary $M\times N$  matrices. If $m$ is a  positive integer such that 
$m\le \sum_{i=1}^sn_i$, then
    \begin{eqnarray*}
 &&\sum_{\sum_{i=1}^sk_i=m\atop k_1\le n_1,k_2\le n_2,\ldots, 
k_s\le n_s}{n_1\choose k_1}
{n_2\choose k_2}\cdots {n_s\choose k_s}
(\sum_{i=1}^sk_i{\rm{\bf A}}_i)(\sum_{i=1}^sk_i{\rm{\bf A}}_i^*)\\
(30) &=&{\sum_{i=1}^s n_i -2\choose m-1}\sum_{i=1}^s n_i
{\rm{\bf A}}_i{\rm{\bf A}}_i^*  
+{\sum_{i=1}^s n_i -2\choose m-2}
(\sum_{i=1}^sn_i{\rm{\bf A}}_i)(\sum_{i=1}^sn_i{\rm{\bf A}}_i^*),
 \qquad\qquad\qquad
   \end{eqnarray*}
 where the summation ranges over all nonnegative integers 
$k_i$  $(i=1,2,\ldots , s)$ such that $k_1\le n_1,k_2\le n_2,\ldots, 
k_s\le n_s$ and  $\sum_{i=1}^sk_i=m$.
  \end{identity}

\end{document}